\documentclass[psamsfont,a4paper,12pt]{amsart}

\usepackage[english]{babel}   

\usepackage{setspace}     
\usepackage{enumerate} 	
\usepackage{afterpage} 	
\usepackage[usenames,dvipsnames]{color}
\usepackage{graphicx}	

\usepackage{amsmath} 	
\usepackage{amssymb}	
\usepackage{mathrsfs} 	
\usepackage{amsfonts} 	
\usepackage{latexsym} 	
\usepackage[latin1]{inputenc} 

\usepackage{amsthm} 	
\usepackage{stackrel}       
\usepackage{amscd} 	

\usepackage{tabularx}	
\usepackage{longtable}	

\usepackage{verbatim} 
\usepackage{blindtext} 

\usepackage{lipsum}

\allowdisplaybreaks

\newcommand{\C}{\mathbb{C}}

\newcommand{\R}{\mathbb{R}} 
 
\newcommand{\NN}{\mathbb{N}} 
 
\newcommand{\Q}{\mathbb{Q}}

\renewcommand{\to}{\longrightarrow}

\newtheorem{Thm}{Theorem}[section]		

\newtheorem{Cor}[Thm]{Corollary}

\newtheorem*{Ex}{Example}

\theoremstyle{definition}
\newtheorem{Definition}[Thm]{Definition}

\theoremstyle{remark}
 
\newtheorem*{rmk}{Remark}

\newtheorem*{ex}{Example}

\newtheorem{ind}[]{{\rm\it Indice}}

\usepackage{multicol}

\usepackage{tikz}
\usetikzlibrary{arrows}

\usepackage{paralist}
\usepackage{cite}

\usepackage{breqn}   

\title{The Jensen-P\'olya program for various $L$-functions}

\author[Wagner]{Ian Wagner}

\begin{document}
\numberwithin{equation}{section}

\begin{abstract}
P\'{o}lya proved in 1927 that the Riemann hypothesis is equivalent to the hyperbolicity of all of the Jensen polynomials of degree $d$ and shift $n$ for the Riemann Xi-function.  Recently, Griffin, Ono, Rolen, and Zagier proved that for each degree $d \geq 1$ all of the Jensen polynomials for the Riemann Xi-function are hyperbolic except for possibly finitely many $n$.  Here we extend their work by showing the same statement is true for suitable $L$-functions.  This offers evidence for the generalized Riemann hypothesis.
\end{abstract}

\maketitle

\section{Introduction and statement of results}
By extending notes of Jensen, P\'{o}lya \cite{P} proved that the Riemann hypothesis (RH) is equivalent to the hyperbolicity of the \textit{Jensen polynomials} for Riemann's Xi-function.  The Riemann Xi-function is the entire function that shifts the zeros of the Riemann zeta-function, $\zeta(s)$, from the line with real part $\frac{1}{2}$ to the real line.  It is given by 
\begin{equation}
\Xi(z) := \frac{1}{2} \left( -z^{2} - \frac{1}{4} \right) \pi^{\frac{iz}{2} - \frac{1}{4}} \Gamma \left(-\frac{iz}{2}+\frac{1}{4} \right) \zeta \left(-iz + \frac{1}{2} \right),
\end{equation}
where $\Gamma(s)$ is the gamma function.  We can consider a change of variable and define the coefficients $\gamma(n)$ by the Taylor expansion of this new function:
\begin{equation}
\Xi_{1}(x) = 8 \cdot \Xi \left( i \sqrt{x} \right) =: \sum_{n \geq 0} \frac{\gamma(n)}{n!} \cdot x^{n}.
\end{equation}
P\'{o}lya originally proved that RH is equivalent to $\Xi_{1}$ having an infinite product expansion of the form $\Xi_{1}(x) = ce^{\sigma x} \prod_{n \geq 1} \left( 1 + \frac{x}{x_{n}} \right)$, where $c$ is a constant, $\sigma \geq 0$, $x_{n} \in \R^{+}$, and $\sum x_{n}^{-1} < \infty$.  This condition can be encoded by the hyperbolicity of Jensen polynomials.

We say that a polynomial $f \in \R[X]$ is \textit{hyperbolic} if all of its roots are real.  Given a sequence $a: \NN \to \R$ and positive integers $d$ and $n$, the associated \textit{Jensen polynomial of degree $d$ and shift $n$} is defined by
\begin{equation}
J^{d,n}_a(X) := \sum_{j=0}^d {d \choose j} a(n+j) X^j.
\end{equation}
RH is equivalent to the hyperbolicity of $J_{\gamma}^{d,n}(X)$ for all $d$ and $n$ and where $\gamma$ is given in equation ($1.2$) as the Taylor coefficients of $\Xi_{1}(x)$ \cite{CV, DL, P}.  The historical context of this approach to RH and a commentary on the results of \cite{GORZ} is given in \cite{Bo}.  Due to the difficulty of proving RH, research before \cite{GORZ} focused on establishing hyperbolicity for all shifts $n$ for small $d$.  Work of Csordas, Norfolk, and Varga and Dimitrov and Lucas \cite{CNV, DL} shows that $J_{\gamma}^{d,n}(X)$ is hyperbolic for all $n$ when $d \leq 3$.  In \cite{GORZ}, Griffin, Ono, Rolen, and Zagier prove that for any $d \geq 1$, $J_{\gamma}^{d,n}(X)$ is hyperbolic with at most finitely exceptions $n$.  They prove this by showing that for a fixed $d$,
\begin{equation*}
\lim_{n \to \infty} J^{d,n}_{\gamma}(\alpha(n) X + \beta(n)) = H_{d}(X),
\end{equation*}
where $H_{d}(X)$ is the $d$-th Hermite polynomial and $\alpha(n)$ and $\beta(n)$ are certain sequences.  The Hermite polynomials are known to have real distinct roots so $J_{\gamma}^{d,n}(X)$ must also eventually have real distinct roots.  In fact, Griffin, Ono, Rolen, and Zagier show that there is a family of sequences whose Jensen polynomials share the same property.  
\begin{Definition} \label{growth}
A real sequence $a: \NN \to \R$ is \textbf{Hermite-Jensen} if there exists sequences of positive real numbers $\{A(n)\}$ and $\{\delta(n) \}$ with $\delta(n)$ tending to zero, which satisfy
\begin{equation}
\log \left( \frac{a(n+j)}{a(n)} \right) = A(n) j - \delta(n)^{2} j^{2} + o(\delta(n)^d) \qquad \text{as } \ n \to \infty
\end{equation}
for some $d \geq 1$ and all $0 \leq j \leq d$.
\end{Definition}
\begin{rmk}
In \cite{GORZ} the authors give a more general statement about the asymptotic behavior needed for the Jensen polynomials of a sequence to converge to other families of polynomials.
\end{rmk}
In order to show that the Taylor coefficients of Riemann's Xi-function are Hermite-Jensen, an arbitrary precision asymptotic formula for the derivatives $\Xi^{(2n)}(0)$ was found in \cite{GORZ}.
To extend the results in \cite{GORZ} we show that any \textit{good} Dirichlet series is Hermite-Jensen.
\begin{Definition} \label{HJ}
A Dirichlet series $L(s) = \sum_{n \geq 1} a(n) n^{-s}$ is \textbf{good} if the following hold.
\begin{enumerate}
\item $L(s)$ has a completed form, $\Lambda(s)$ given in equation \eqref{completed}, that has an integral representation of the form
\begin{equation*}
\Lambda(s)=N^{\frac{s}{2}}  \int_{0}^{\infty} \left[ f(t) - f(\infty) \right] t^{s} \frac{dt}{t},
\end{equation*}
where the function $f(t)$ has the form
\begin{equation*}
f(t) = \alpha(0) + \sum_{n \geq n_{0}} \alpha(n) e^{-\pi nt},
\end{equation*}
where $f(\infty) = \alpha(0)$.
\item The function $f(t)$ satisfies
\begin{equation*}
f \left( \frac{1}{Nt} \right) = \epsilon N^{\frac{k}{2}} t^{k} f(t),
\end{equation*}
where $\epsilon \in \{ \pm 1 \}$ which gives rise to an analytic continuation and a functional equation $\Lambda(s) = \epsilon \Lambda(k-s)$ for some $k \in \Q$.
\item The coefficients of $\Lambda(s)$ are real.
\end{enumerate}

\end{Definition}
For a good Dirichlet series $L(s)$, we define  \begin{equation}
\Xi(z) := \begin{cases} \left(-z^{2} - \frac{k^{2}}{4} \right) \Lambda \left( \frac{k}{2} -iz \right)  & \text{if } \Lambda(s) \ \text{has a pole at } s =k \\
\Lambda \left( \frac{k}{2} -iz \right) & \text{otherwise}.
\end{cases} 
\end{equation}
If $\Lambda(s)=\Lambda(k-s)$, then we define
\begin{equation}
\Xi_{1}(x):= \Xi(i \sqrt{x}) =: \sum_{n \geq 0} \frac{\gamma(n)}{n!} x^{n},
\end{equation}
where 
\begin{equation*}
\gamma(n) = (-1)^n \frac{n!}{(2n)!} \cdot \Xi^{(2n)}(0).
\end{equation*}
If $\Lambda(s)= -\Lambda(k-s)$, then define
\begin{equation}
\Xi_{1}(x) := \frac{\Xi(i \sqrt{x})}{\sqrt{x}} =: \sum_{n \geq 0} \frac{\gamma(n)}{n!} x^{n},
\end{equation}
where 
\begin{equation*}
\gamma(n) = i^{2n+1} \frac{n!}{(2n+1)!} \cdot \Xi^{(2n+1)}(0).
\end{equation*}
\begin{Thm} \label{Main}
Suppose that $L(s)$ is a good Dirichlet series.  Then $J_{\gamma}^{d,n}(X)$ is hyperbolic with at most finitely many exceptions $n$ for each fixed $d \geq 1$.
\end{Thm}
\begin{rmk}
This offers evidence for the generalized Riemann Hypothesis (GRH).
\end{rmk}
\begin{rmk}
Notice that in order to study the Jensen polynomials associated to $\gamma(n)$ we must understand the derivatives of $\Xi(z)$ at $z=0$, or equivalently the derivatives of $\Lambda(s)$ at $s = \frac{k}{2}$.  In order to prove Theorem \ref{Main} we will prove an asymptotic formula with arbitrary precision for these derivatives.
\end{rmk}
\begin{rmk}
All good $L$-series satisfy the \textit{Gaussian Unitary Ensemble} (GUE) random matrix prediction in derivative aspect.  Dyson, Montgomery, and Odlyzko \cite{DY, M, O} conjectured that the non-trivial zeros of the Riemann zeta function and other suitable $L$-functions are distributed like the eigenvalues of random Hermitian matrices.  These eigenvalues and the roots of the suitably normalized Hermite polynomials, $H_{d}(X)$, as $d \to \infty$ both satisfy Wigner's Semicircular Law (Chapter 3 of \cite{AGZ}).  The roots of $J_{\gamma}^{d,0}(X)$, as $d \to \infty$, approximate the zeros of $\Lambda \left( \frac{k}{2} -iz \right)$ \cite{P} so these roots are also expected to satisfy Wigner's Semicircular Law.  The derivatives of the completed $L$-function  are also predicted to satisfy GUE and higher derivatives correspond to $n$ growing in $J_{\gamma}^{d,n}(X)$ so it is natural to study $J_{\gamma}^{d,n}(X)$ as $n \to \infty$.  For a good $L$-function the $J_{\gamma}^{d,n}(X)$ converge to the Hermite polynomials which satisfy GUE in degree aspect.  This is what is meant by the statement that good $L$-functions satisfy GUE in derivative aspect.
\end{rmk}
The following corollaries give some examples of Hermite-Jensen Dirichlet series.
\begin{Cor} \label{Dir}
Dirichlet $L$-functions for real primitive self-dual characters are good.
\end{Cor}
\begin{Cor} \label{Mod}
Let $f \in S^{new}_{2k}(\Gamma_{0}(N))$ be a weight $2k$ modular newform on $\Gamma_{0}(N)$, then the modular $L$-function associated to $f$ is good.
\end{Cor}
\begin{Cor} \label{Ded}
The Dedekind zeta-function for a number field is good.
\end{Cor}
In each of these cases we prove an arbitrary precision asymptotic formula for the derivatives of the completed $L$-series at its central value.  We do this to show that these $L$-series are Hermite-Jensen, but these results are also of independent interest.
This paper is organized in the following way.  In Section 2 we will prove Theorem \ref{Main}.  In Section 3 we will prove the three corollaries.  This section will include the asymptotic formulae for the derivatives of the $L$-functions mentioned above.

\section*{Acknowledgements}
The author would like to thank Larry Rolen, Michael Griffin, and Ken Ono for helpful discussions related to this work.

\section{Asymptotics for $\Xi^{(n)}(0)$.}

Let $L(s) = \sum_{n \geq 1} a(n) n^{-s}$ be a good Dirichlet series.  We thus know that $L(s)$ has a completed form
\begin{align} \label{completed}
\begin{split}
\Lambda(s) &= N^{\frac{s}{2}} \prod_{j=1}^{J} \Gamma_{\R}(s) \prod_{m=1}^{M} \Gamma_{\C}(s) \cdot L(s) \\
&= N^{\frac{s}{2}} \int_{0}^{\infty} \left[ f(t) - f(\infty) \right] t^{s} \frac{dt}{t},
\end{split}
\end{align}
where $\Gamma_{\R}(s) :=\pi^{-\frac{s}{2}} \Gamma \left( \frac{s}{2} \right)$, $\Gamma_{\C}(s) := 2 (2 \pi)^{-s} \Gamma(s)$, and $\Gamma(s):= \int_{0}^{\infty} e^{-t} t^{s-1} dt$ is the usual gamma function.
Because of the transformation properties of $f(t)$, we split the integral at $\frac{1}{\sqrt{N}}$ to arrive at
\begin{equation}
\Lambda(s) = \frac{(\epsilon s-s+k) f(\infty)}{s(s-k)} + \int_{\frac{1}{\sqrt{N}}}^{\infty} \left(f(t) -f(\infty)\right) \left( \epsilon N^{\frac{k-s}{2}} t^{k-s}  +N^{\frac{s}{2}} t^{s}\right) \frac{dt}{t}.
\end{equation}
We have the following expression for the derivatives of $\Lambda(s)$: 
\begin{align*}
\Lambda^{(n)}(s) &= \frac{[(-1)^{n+1} (k-s)^{n+1} - \epsilon s^{n+1}]f(\infty) n!}{s^{n+1} (k-s)^{n+1}} \\
&+ \int_{\frac{1}{\sqrt{N}}}^{\infty} (f(t)-f(\infty)) \left(N^{\frac{s}{2}}t^{s} + (-1)^{n} \epsilon N^{\frac{k-s}{2}} t^{k-s} \right) \left(\frac{1}{2} \log(N) + \log(t) \right)^{n} \frac{dt}{t}.
\end{align*}
At $s=\frac{k}{2}$ and $z=0$ we have
\begin{equation}
\Lambda^{(n)} \left(\frac{k}{2} \right) = \frac{2^{n+1} f(\infty) n! ((-1)^{n+1} - \epsilon)}{k^{n+1}} + F(n)
\end{equation}
and 
\begin{equation}
\Xi^{(n)}(0) = \begin{cases} (-i)^{n} \frac{8 \binom{n}{2} F(n-2) - k^{2} F(n)}{4} & \text{if } \ \Lambda(s) \ \text{has a pole at} \ s=k \\
(-i)^{n}F(n) & \text{otherwise}, \end{cases}
\end{equation}
where
\begin{equation} \label{F}
F(n) = \frac{1}{2^n} \int_{\frac{1}{\sqrt{N}}}^{\infty} (f(t) - f(\infty)) N^{\frac{k}{4}} t^{\frac{k}{2} -1}(1+(-1)^{n} \epsilon) \left( \log(N) + 2 \log(t) \right)^{n} dt
\end{equation}
for all $n \geq 0$.  The large asymptotics of $\Lambda^{(n)} \left( \frac{k}{2} \right)$ and $\Xi^{(n)}(0)$ are obtained from the following theorem.
\begin{Thm} \label{Asy}
The function $F(n)$ defined in (\ref{F}) is given to all orders in $n$ by the asymptotic expansion
\begin{align}
F(n) &\sim \frac{\alpha(n_{0}) \sqrt{2 \pi} N^{\frac{k}{4}} (1 + (-1)^{n} \epsilon)}{2^n} \frac{L^{n+1}}{\sqrt{\left(1+ \frac{L}{2} \right)n - \left(\frac{k}{2} -1 \right)L^2}} \\
& \times e^{\frac{k}{4}(L-\log(N)) - \frac{2n}{L} -\frac{k}{2} +1} \left(1 + \frac{b_{1}}{n} + \frac{b_{2}}{n^2} + \cdots \right) \qquad (n \to \infty), \nonumber
\end{align}
where $L = L(n) \approx 2 \log \left(\frac{n \sqrt{N}}{\log(n \sqrt{N})} \right)$ is the unique positive solution of the equation $n = \frac{1}{2} \left( \pi n_{0} e^{\frac{L-\log(N)}{2}} -\frac{k}{2} +1 \right)L$ and each coefficient $b_{k}$ belongs to $\Q(L)$, the first value being $b_{1} = \frac{2(31L^4 + 189L^3 +542L^2 +744L +496)}{3(L+2)^3}$.
\end{Thm}

\begin{proof}[Proof of Theorem \ref{Asy}]
We approximate the integrand in (\ref{F}) by the function 
\begin{equation*}
g(t) = \alpha(n_{0})e^{- \pi n_{0} t} N^{\frac{k}{4}} (1 + (-1)^{n} \epsilon) t^{\frac{k}{2} -1} (\log(N) + 2\log(t))^{n}.
\end{equation*}
From now on we let $n$ be fixed and will omit it from our notations.  We have that $t \frac{d}{dt}(\log(g(t))) = \frac{2n}{\log(N) +2\log(t)} - \pi n_{0} t + \frac{k}{2} -1$, so $g(t)$ assumes its unique maximum at $t=a$ where $a$ is the solution in $\left(\frac{1}{\sqrt{N}}, \infty \right)$ of 
\begin{equation*}
n = \frac{1}{2} \left( \pi n_{0} a - \frac{k}{2} +1 \right)(\log(N) + 2\log(a)).
\end{equation*}
For convenience, we define $L = \log(N) + 2 \log(a)$ so we have 
\begin{equation*}
n = \frac{1}{2} \left( \pi n_{0} e^{\frac{L-\log(N)}{2}} - \frac{k}{2} +1 \right)L.
\end{equation*}
We can then use Lambert's $W$ function to asymptotically solve this equation.  Lambert's $W$ function is defined as the solution to $z= W(z)e^{W(z)}$.  It has the nice property that $Y=Xe^{X}$ if and only if $X=W(Y)$.  If we take a branch cut to restrict $W$ to be real valued, then we have that the principal branch has a Taylor series around $0$ given by $W(x) = \sum_{n=1}^{\infty} \frac{(-n)^{n-1}}{n!} x^{n}$.  For large $x$, $W(x)$ is asymptotic to $W(x) = \ln(x) - \ln \ln(x) + \frac{\ln \ln(x) (\ln \ln(x) -2)}{\ln^{2}(x)} + O \left(\left(\frac{\ln \ln(x)}{\ln(x)} \right)^{3} \right)$ \cite{CGHJK}.  Therefore, we have $L \approx 2 \log \left( \frac{n \sqrt{N}}{\log(n \sqrt{N})} \right)$.  We now follow \cite{GORZ} and apply the saddle point method.  The Taylor expansion of $g(t)$ around $t=a$ is given by
\begin{equation*}
\frac{g((1+\lambda)a)}{g(a)} = \left( 1 + \frac{2 \log(1+\lambda)}{\log(N) + 2 \log(a)} \right)^{n} (1 + \lambda)^{\frac{k}{2} -1} e^{- \pi n_{0} \lambda a} =e^{-\frac{C \lambda^{2}}{2}} \left(1 + A_{3} \lambda^3 + A_{4} \lambda^{4} + \cdots \right),
\end{equation*}
where $C = n \left( \frac{\varepsilon}{2} + \varepsilon^{2} \right) + \frac{k}{8} -\frac{1}{4}$, $\varepsilon = \frac{1}{\log(N) + 2 \log(a)} = L^{-1}$, and the $A_{i}$ are polynomials of degree $\lfloor i/3 \rfloor$ in $n$ with coefficients in $\Q[\varepsilon]$.  This expansion was found by expanding $\log(g((1 + \lambda)a)) - \log(g(a))$ in $\lambda$.  The linear term vanises by choice of $a$ and the quadratic term is $-\frac{C \lambda^2}{2}$.  The coefficients of the higher powers of $\lambda$ are all linear expressions in $n$ with coefficients in $\Q[\varepsilon]$.  Exponentiating this expansions gives our expression for $g((1+\lambda)a)/g(a)$.  The important behavior is that the dominant term of each $A_{i}$ comes primarily from the exponential of the cubic term of the logarithmic expansion.  The first few $A_{i}$ are 
\begin{align*}
A_{3} &= 2n \left( \frac{\varepsilon}{3} + \varepsilon^{2} + \frac{4 \varepsilon^{3}}{3} \right) + \frac{k}{6} -\frac{1}{3}, \\
A_{4} &=- n \left( \frac{\varepsilon}{2} + \frac{11 \varepsilon^2}{6} + 4 \varepsilon^3 + 4 \varepsilon^4 \right) - \frac{k}{8} + \frac{1}{4}, \\
A_{5} &= n \left( \frac{2 \varepsilon}{5} + \frac{5 \varepsilon^2}{3} + \frac{14 \varepsilon^3}{3} + 8 \varepsilon^4 + \frac{32 \varepsilon^5}{5} \right) + \frac{k}{10} - \frac{1}{5}, \\
A_{6} &= n^{2} \left( \frac{2 \varepsilon^2}{9} + \frac{4 \varepsilon^3}{3} + \frac{34 \varepsilon^4}{9} + \frac{16 \varepsilon^5}{3} + \frac{32 \varepsilon^6}{9} \right) + \frac{k^2 -7k +10}{36} \\
&+ n \left( \frac{(10k-50) \varepsilon}{90} + \frac{(30k - 197) \varepsilon^2}{90} + \frac{(40k-530) \varepsilon^3}{90} - \frac{34 \varepsilon^4}{3}  - 16 \varepsilon^5  - \frac{32 \varepsilon^6}{3} \right). 
\end{align*}
We plug in $t = (1 + \lambda)a$ to arrive at the asymptotic expansion
\begin{align*}
\frac{1}{2^n} \int_{\frac{1}{\sqrt{N}}}^{\infty} g(t) dt &= \frac{a g(a)}{2^n} \int_{-1 +\frac{1}{a \sqrt{N}}}^{\infty} e^{-\frac{C \lambda^2}{2}} \left(1 + A_{3} \lambda^3 + A_{4} \lambda^4 + \cdots \right) d \lambda \\
&= \frac{a g(a)}{2^n} \sqrt{ \frac{2 \pi}{C}} \left( 1 + \frac{3A_{4}}{C^2} + \frac{15A_{6}}{C^3} + \cdots + \frac{(2i-1)!! A_{2i}}{C^i} +\cdots \right).
\end{align*}
This expression and the one in Theorem \ref{Asy} are interpreted as asymptotic expansions.  These series do not converge for a fixed $n$, but we can truncate the approximation at $O(n^{-A})$ for some $A>0$, and as $n \to +\infty$ this approximation becomes true to the given precision.  We substitute the formulas for $C$ and $A_{i}$ in terms of $n$ in order to obtain the statement in the theorem.  We also replace $F(n)$ by the integral over $g(t)$ with only the $A_{2i}$ with $i \leq 3k$ contributing to $b_{k}$.  The same asymptotic formula will hold with this replacement because the ratio of $g(t)$ and the integrand of $F(n)$ is equal to $1 + O(n^{-K})$ for any $K>0$ for $t$ near $a$.

\end{proof}
\section{Proof of Theorem \ref{Main}}
Our goal is to show that $\{\gamma(n) \}$ satisfies the growth conditions of Definition \ref{growth}. Recall from Section 1 that 
\begin{equation}
\gamma(n) = \begin{cases} (-1)^{n} \frac{n!}{(2n)!} \Xi^{(2n)}(0) & \text{if } \ \epsilon =1 \\
i^{2n+1} \frac{n!}{(2n+1)!} \Xi^{(2n+1)}(0) & \text{if } \ \epsilon=-1, \end{cases}
\end{equation}
where $\gamma(n)$ are the Taylor coefficients of $\Xi_{1}(x)$.  Therefore, if we have
\begin{equation} \label{Fhat}
\widehat{F}(n) = \frac{\alpha(n_{0}) \sqrt{2 \pi} N^{\frac{k}{4}} (1 + (-1)^{n} \epsilon) L^{n+1}}{2^{n} \sqrt{\left(1 + \frac{L}{2} \right)n - \left( \frac{k}{2} -1 \right)L^2}}  e^{\frac{k}{4} \left( L-\log(N) \right) - \frac{2n}{L} - \frac{k}{2} +1} \left(1 + \frac{b_{1}}{n} \right)
\end{equation}
and
\begin{equation*}
\widehat{\Xi}^{(n)}(0) = \left\{\!\begin{aligned}
&(-i)^{n} 2 \binom{n}{2} \widehat{F}(n-2) & \text{if } \ \Lambda(s) \ \text{has a pole} \\
&(-i)^{n} \widehat{F}(n) & \text{otherwise}
\end{aligned}\right\} = \Xi^{(n)}(0) \cdot \left( 1 + O \left(\frac{1}{n^{2-\varepsilon}} \right) \right)
\end{equation*}
as in the example above, then 
\begin{equation}
\widehat{\gamma}(n) = \left\{\!\begin{aligned}
&\frac{n!}{(2n-2)!} \widehat{F}(2n-2) & \text{if } \ \Lambda(s) \text{ has a pole and } \ \epsilon =1 \\
&\frac{n!}{(2n-1)!} \widehat{F}(2n-1) & \text{if } \ \Lambda(s) \ \text{ has a pole and } \ \epsilon =-1 \\
&\frac{n!}{(2n)!} \widehat{F}(2n) & \text{if } \Lambda(s) \ \text{ does not have a pole and } \epsilon=1 \\
&\frac{n!}{(2n+1)!} \widehat{F}(2n+1) & \text{if } \Lambda(s) \ \text{ does not have a pole and } \epsilon=-1
\end{aligned}\right\} = \gamma(n) \cdot \left(1 + O \left(\frac{1}{n^{2-\varepsilon}} \right) \right).
\end{equation}
We will show that the $\gamma(n) = \frac{n!}{m!} \widehat{F}(m)\cdot \left( 1 + O \left( \frac{1}{n^{2-\varepsilon}} \right) \right)$ for $m= 2n-2, 2n-1, 2n$, or $2n+1$ form a Hermite-Jensen sequence.  Recall that
\begin{equation*}
b_{1} = \frac{2(31L^4 + 189L^3 + 542L^2 +744L+496)}{3(L+2)^3}.
\end{equation*}
Using Stirling's approximation $r! = \sqrt{2 \pi r} \left( \frac{r}{e} \right)^{r} \cdot \left( 1 + \frac{1}{12r} \right) \cdot \left(1 + O(1/r^2) \right)$, we have
\begin{align} \label{gamma}
\gamma(n) &= \frac{\alpha(n_{0}) N^{\frac{k}{4}} ( 1+ (-1)^m \epsilon) e^{m-n} n^{n+\frac{1}{2}} \left(1 + \frac{1}{12n} \right) L(m)^{m}}{2^{m} m^{m+\frac{1}{2}} \left(1 + \frac{1}{12m} \right)} \cdot \sqrt{\frac{2 \pi}{C(m)}} \\
& \times {\rm{exp}} \left( \frac{k}{4}(L(m) - \log(N)) - \frac{2m}{L(m)} -\frac{k}{2} +1 \right) \left(1 + \frac{b_{1}(m)}{m} \right) \cdot \left(1 +O \left(\frac{1}{n^{2 - \varepsilon}} \right) \right). \nonumber
\end{align}
Recall that $L(m)$ and $b_{1}(m)$ are given in Theorem \ref{Asy} and $C(m) = m \left(\frac{L(m)^{-1}}{2} + L(m)^{-2} \right) + \frac{k}{8}- \frac{1}{4}$.  $L(m)$ can be viewed as a holomorphic and non-vanishing function for ${\rm{Re}}(m)>0$, so we have a Taylor expansion in $j$ for the ratio $L(m+2j)/L(m)$ given by
\begin{equation}
\mathcal{L}(j;m):= \frac{L(m+2j)}{L(m)} = 1 + \sum_{r \geq 1} \frac{\ell_{r}(m)}{r!} \cdot j^{r}
\end{equation} 
which converges when $|j|< \frac{m}{2}$, so we will assume this throughout the proof.  If we let $J = \frac{\lambda m}{2}$ for some $-1<\lambda<1$, then the asymptotic $L(m) \approx \log \left( \frac{\sqrt{N} m}{\log(\sqrt{N}m)} \right)$ gives the limit
\begin{equation*}
\lim_{m \to \infty} \mathcal{L}(J;m) = \lim_{m \to \infty} \frac{L(m(\lambda +1))}{L(m)} =1.
\end{equation*}   
This implies that $\ell_{r}(m) = o \left( \left( \frac{2}{m} \right)^r \right)$.  If we expand 
\begin{equation*}
m +2j = \frac{L(m) \cdot \mathcal{L}(j;m)}{2} \left( \pi n_{0} e^{\frac{L(m) \cdot \mathcal{L}(j;m) - \log(N)}{2}} - \frac{k}{2} +1 \right)
\end{equation*}
in $j$ then we find $\ell_{1}(m) = \frac{8}{4m(L/2 +1) + L^{2}(k/2 -1)} = \frac{2}{C \cdot L^2}$ and $\ell_{2}(m) = \frac{-(L/2 +2)(m + kL/4 - L/2)}{C^{3} \cdot L^{5}}$, where $L=L(m)$ and $C=C(m)$.  We will also define
\begin{equation*}
\mathcal{C}(j;m) := \frac{C(m+2j)}{C(m)} = 1 + \sum_{r \geq 1} \frac{c_{r}(m)}{r!} \cdot j^{r},
\end{equation*}
and 
\begin{equation*}
\mathcal{B}(j;m) := \frac{1 + \frac{b_{1}(m+2j)}{m+2j}}{1 + \frac{b_{1}(m)}{m}} = 1 + \sum_{r \geq 1} \frac{\beta_{r}(m)}{r!} \cdot j^{r}.
\end{equation*}
We have the limits 
\begin{equation*}
\lim_{m \to \infty} \mathcal{C}(J;m) = 1 + \lambda \ \text{ and } \ \lim_{m \to \infty} \frac{m}{2} \left( \mathcal{B}(J;m) -1 \right) =0,
\end{equation*}
which imply $c_{r}(m) = o \left( \left(\frac{2}{m} \right)^r \right)$ and $\beta_{r}(m) = o \left( \left(\frac{2}{m} \right)^{r+1} \right)$.  Using the expansion for $\mathcal{L}(j;m)$ and the expression for $\ell_{1}(m)$ we can find that $c_{1}(m) = \frac{L+2}{C \cdot L^2} - \frac{m(L+4)}{C^2 \cdot L^4}$.  
Define $R_{\gamma}(j;m) := \frac{\widehat{\gamma}(n+j)}{\widehat{\gamma}(n)}$, then after some manipulations we have
\newpage
\begin{align*}
R_{\gamma}(j;m) &= \frac{e^{j} n^{j} L(m)^{2j}}{2^{2j} m^{2j}} \cdot \frac{ \left( \frac{n+j}{n} \right)^{n+j +\frac{1}{2}}}{ \left( \frac{m+2j}{m} \right)^{m+2j+\frac{1}{2}}} \frac{\left(1 + \frac{1}{12(n+j)} \right) \left( 1 + \frac{1}{12m} \right)}{\left(1 +\frac{1}{12n} \right) \left(1 +\frac{1}{12(m+2j)} \right)} \cdot \frac{\mathcal{L}(j;m)^{m+2j}}{\sqrt{\mathcal{C}(j;m)}} \\
& \times {\rm{exp}} \left( \frac{kL(m)}{4}(\mathcal{L}(j;m) -1) - \frac{2(m+2j)}{\mathcal{L}(j;m) L(m)} + \frac{2m}{L(m)} \right) \cdot \mathcal{B}(j;m).
\end{align*}
By equation (\ref{gamma}), for $j$ fixed and as $m \to \infty$ (and thus $n \to \infty$), we have
\begin{equation}
\frac{\gamma(n+j)}{\gamma(n)} = R_{\gamma}(j;m) \cdot \left(1 + O \left(\frac{1}{n^{2 - \varepsilon}} \right) \right).
\end{equation}
Notice that the first factor in $R(j;m)$ is the $j$th power of $ \frac{enL(m)^2}{2^2 m^2}$.  This factor will essentially be $e^{A(n)}$.  We will now look at the expansion
\begin{equation*}
\log R(j;m) =: \sum_{r \geq 1} g_{r}(m) j^{r}.
\end{equation*}
We again let $J=\frac{m \lambda}{2}$ for $-1<\lambda<1$, then we have
\begin{equation*}
\lim_{m \to \infty} \frac{2 \log R(J;m)}{m} = \sum_{r \geq 1} g_{r}(m) \left(\frac{m}{2} \right)^{r-1} \lambda^{r} =-(\lambda +1) \log(\lambda +1),
\end{equation*}
which tells us that $g_{r}(m) = O \left( \left(\frac{m}{2} \right)^{1-r} \right)$.  We can use our previous expansions and the formula for $R_{\gamma}(j;m)$ to find 
\begin{align*}
g_{1}(m) &= \log \left(\frac{n L^{2}}{4 m^2} \right) +m \cdot \ell_{1}(m) \left(\frac{L+2}{L} \right) - \frac{4}{L} +\frac{k \cdot \ell_{1}(m) \cdot L}{4} - \frac{c_{1}(m)}{2} + O \left( \frac{1}{n^{2 - \varepsilon}} \right), \\
g_{2}(m) &= -\frac{2}{m} + (4 \ell_{1}(m) + m \cdot \ell_{2}(m))\left(\frac{L+2}{2L} \right) -m \cdot \ell_{1}(m)^{2} \left(\frac{L+4}{2L} \right) + O \left( \frac{1}{n^{2 - \varepsilon}} \right).
\end{align*}
We can use the formulas for $\ell_{1}(m)$ and $\ell_{2}(m)$ to simplify these to
\begin{align}
g_{1}(m) &= \log \left( \frac{n L^2}{4 m^2} \right) + \frac{L -2}{2C \cdot L^2} + \frac{m(L+4)}{2C^2 \cdot L^4}  + O \left( \frac{1}{n^{2 - \varepsilon}} \right), \\
g_{2}(m) &= -\frac{2}{m} + \frac{4}{C \cdot L^{2}} + O \left( \frac{1}{n^{2 - \varepsilon}} \right). 
\end{align}
We now let 
\begin{align}
A(n) &= \log \left( \frac{n L^2}{4 m^2} \right)+ \frac{L -2}{2C \cdot L^2} + \frac{m(L+4)}{2C^2 \cdot L^4} \\
\delta(n) &= \sqrt{\frac{2}{m} - \frac{4}{C \cdot L^2}}.
\end{align}
These functions satisfy the conditions of Definition \ref{growth} for the sequence $\{ \gamma(n) \}$.  The fact that $\delta(n) \to 0$ follows from the asymptotics given above and the precision of $O \left( \frac{1}{n^{2-\varepsilon}} \right)$ satisfies the necessary growth conditions given in Definition \ref{growth}.  
\newpage
\section{Proofs of Corollaries}

\subsection{Dirichlet $L$-functions}
\subsubsection{Proof of Corollary \ref{Dir}}
 Let $\chi$ be a Dirichlet character modulo $N>1$.  Then we define the Dirichlet L-function as 
\begin{equation} \label{L}
L(\chi, s) := \sum_{n \geq 1} \frac{\chi(n)}{n^s}
\end{equation}
for ${\rm{Re}}(s) >1$.  If we let $\chi$ be the trivial character, then our $L$-function is the Riemann zeta function.  This case was handled in \cite{GORZ}.
Next, recall the twisted theta function
\begin{align} \label{theta}
\theta_{\chi}(z)= \chi(0) + 2\sum_{n \geq 1} \chi(n) n^{\nu} e^{\pi i n^{2}z}
\end{align}
where $\nu=0$ if $\chi$ is even and $\nu=1$ if $\chi$ is odd.  The twisted theta function satisfies the functional equation
\begin{equation} \label{thetafun}
\theta_{\chi}(z) = \frac{\tau(\chi)}{i^{\nu} \sqrt{N} (-iNz)^{\frac{1}{2} +\nu}} \theta_{\bar{\chi}} \left(\frac{1}{N^{2}z} \right)
\end{equation}
where $\tau(\chi)$ is a Gauss sum and $\bar{\chi}$ is the dual character.  We will focus on real primitive self-dual characters so we have $\chi = \bar{\chi}$ and $\tau(\chi) = i^{\nu} \sqrt{N}$. Define the completed Dirichlet $L$-function by 
\begin{align}
\Lambda(\chi, s) &:= \left( \frac{N}{\pi} \right)^{\frac{s + \nu}{2}} \Gamma \left( \frac{s + \nu}{2} \right) L(\chi, s) \\
&= \frac{1}{2} N^{\frac{s + \nu}{2}} \int_{0}^{\infty}  \theta_{\chi}(iy) y^{\frac{s+\nu}{2}} \frac{dy}{y}. \nonumber
\end{align}    
Using equation \eqref{thetafun} and the fact that $\chi$ is a real primitive self-dual character, we have the following functional equation
\begin{equation}
\Lambda(\chi, s) = \Lambda(\chi, 1-s).
\end{equation} 
The completed Dirichlet $L$-function has the required integral representation, functional equation, and real coefficients so it is good.

\subsubsection{Derivatives at central values and Dirichlet Jensen polynomials}
We want to study the derivatives of $\Lambda(\chi, s)$ which are given by
\begin{equation}
\Lambda^{(n)} \left(\chi, s \right) = \frac{1}{2^{2n+1}} \int_{\frac{1}{N}}^{\infty} \theta_{\chi}(iy) \left( (Ny)^{\frac{s+\nu}{2}} + (-1)^{n} (Ny)^{\frac{1-s+\nu}{2}} \right) \left(\log(N^2) +2 \log(y) \right)^{n} \frac{dy}{y}.
\end{equation}
At the central value $s= \frac{1}{2}$ we have
\begin{equation}
\Lambda^{(n)} \left( \chi, \frac{1}{2} \right) = \frac{1}{2^{2n +1}} \int_{\frac{1}{N}}^{\infty} \theta_{\chi}(iy) (Ny)^{\frac{1}{4} + \frac{\nu}{2}} (1+(-1)^{n}) (\log(N^2) + 2\log(y))^{n} \frac{dy}{y}.
\end{equation}
Because the Dirichlet $L$-functions fit into our framework we have the following theorem which gives an arbitrary precision asymptotic formula for these derivatives.
\begin{Thm}
Assume the notation above.  The large $n$ asymptotics for $\Lambda^{(n)} \left( \chi, \frac{1}{2} \right)$ and $\Xi^{(n)}(\chi, 0)$ are given to all orders by the asymptotic expansion
\begin{align}
F(n)&\sim \frac{\sqrt{2 \pi} N^{\frac{1}{4} + \frac{\nu}{2}} (1 +(-1)^n)}{2^{2n }} \frac{L^{n+1}}{\sqrt{4n \left(1 +\frac{L}{2} \right) - \left(\frac{3}{4} - \frac{\nu}{2} \right)L^{2}}} \\
\times& e^{\left( \frac{1}{8} + \frac{\nu}{4} \right)(L - \log(N^2)) -\frac{2n}{L} +\frac{3}{4} -\frac{\nu}{2}} \left(1 + \frac{b_{1}}{n} + \frac{b_{2}}{n^2} + \cdots \right) \qquad (n \to \infty), \nonumber
\end{align}
where $L=L(n) \approx 2\log \left(\frac{nN}{\log(nN)} \right)$ is the unique positive solution to $n = \frac{1}{2} \left( \pi e^{\frac{L-\log(N^2)}{2}} +\frac{3}{4} - \frac{\nu}{2} \right)L$ and each coefficient $b_{k}$ belongs to $\Q(L)$, the first value being $b_{1} = \frac{L^{4} +9L^{3} +32L^{2} +24L +16}{24(L+2)^{3}}$.
\end{Thm}
\begin{Ex}
Let $\chi_{4}$ be the odd Dirichlet character of modulus $4$.  Using the two-term approximation $\widehat{F}(n)$ given in equation (\ref{Fhat}) we give some approximations $\widehat{\gamma}_{\chi_{4}}(n)$ in the table below.

\begin{center}
  \begin{tabular}{ | c | c | c | c | }
    \hline
    $n$ & $\widehat{\gamma}_{\chi_{4}}(n)$ & $\gamma_{\chi_{4}}(n)$ & $\gamma_{\chi_{4}}(n)/\widehat{\gamma}_{\chi_{4}}(n)$ \\ \hline
    $10$ & $\approx 8.6123842782 \times 10^{-14}$ & $\approx 8.5921206983 \times 10^{-14}$ & $\approx 0.997647158$  \\ 
    $100$ & $\approx 1.0054943805 \times 10^{-174}$ & $\approx 1.0057597216 \times 10^{-174}$ & $\approx 0.9997361785$ \\ 
    $1000$ & $\approx 1.7838444188 \times 10^{-2350}$ & $\approx 1.7838866878 \times 10^{-2350}$ & $\approx 0.9999763051$ \\ 
    $10000$ & $\approx 1.7271165350 \times 10^{-30650}$ & $\approx 1.7271200653 \times 10^{-30650}$ & $\approx 0.9999979560$ \\ 
    $100000$ & $\approx 8.1291521235 \times 10^{-384416}$ & $\approx 8.1291531304 \times 10^{-384416}$ & $\approx 0.9999998761$ \\ 
   
    \hline
  \end{tabular}
\end{center}
\end{Ex}

In the previous section we showed that the Dirichlet $L$-function $L(\chi, s)$ is good.  Dirichlet $L$-functions have a pole at $s=1$ if $\chi$ is principal so we define
\begin{equation}
\Xi(\chi, z) :=\begin{cases} \left(-z^{2} - \frac{1}{4} \right) \Lambda \left(\frac{1}{2} -iz \right) & \text{if } \ \chi \ \text{ is principal} \\
\Lambda \left( \chi, \frac{1}{2} -iz \right) & \text{else}
\end{cases}
\end{equation}
and 
\begin{equation}
\Xi_{1}(\chi, x) := \Xi(\chi, i \sqrt{x}) = \sum_{n \geq 0} \frac{\gamma_{\chi}(n)}{n!} x^{n}
\end{equation}
where
\begin{equation}
\gamma_{\chi}(n) = (-1)^{n} \frac{n!}{(2n)!} \cdot \Xi^{(2n)}(\chi, 0).
\end{equation}
By Theorem \ref{Main} or by using the asymptotic expansion above we know that if $d \geq 1$, then $J_{\gamma_{\chi}}^{d,n}(X)$ is hyperbolic with at most finitely many exceptions $n$.
\begin{ex}
To exemplify Corollary \ref{Dir} we will again consider the odd Dirichlet character of modulus $4$.  Let $L(n)$ be the unique solution to $n = \frac{L(n)}{2} \left(\pi e^{\frac{L(n) - \log(16)}{2}} + \frac{1}{4} \right)$ and set $L=L(2n)$.  Also set
\begin{align*}
C&=C(2n) = 2n \left(\frac{2}{L} + \frac{4}{L^2} \right) -\frac{1}{4}, \\
A(n) &= \log \left( \frac{L^2}{64n} \right) + \frac{2(L-2)}{CL^2} + \frac{16n(L+4)}{C^2 L^4}, \\
\delta(n) &= \sqrt{\frac{1}{2n} - \frac{8}{CL^2}}.
\end{align*}
The following table demonstrates for $d=2$ and $d=3$ that $\widehat{J}^{d,n}_{\gamma_{\chi_{4}}}(X) = J^{d,n}_{\gamma_{\chi_{4}}} \left( \frac{\delta(n)X -1}{e^{A(n)}} \right)$ converges to $H_{d}(X)$ as $n \to \infty$.  The polynomials have been normalized so that their leading coefficients are $1$.

\begin{center}
  \begin{tabular}{ | c | c | c | }
    \hline
    $n$ & $\widehat{J}^{2,n}_{\chi_{4}}(X)$ & $\widehat{J}^{3,n}_{\gamma_{\chi_{4}}}(X)$ \\ \hline
    $100$ & $\approx X^2 + 0.3332X -1.9985$ & $\approx X^3 + 0.8306X^2 -5.8678X -1.3254$ \\ 
    $1000$ & $\approx X^2 + 0.1136X - 1.9997$ & $\approx X^3 + 0.2839X^2 -5.9847X - 0.4414$ \\ 
    $10000$ & $\approx X^2 + 0.0375X - 1.9999$ & $\approx X^3 +0.0936X^2 -5.9984X -0.1435$ \\ 
    $100000$ & $\approx X^2 + 0.0012X -1.9999$ & $\approx X^3 +0.0304X^2 -5.9998X -0.0444$ \\ 
    $\vdots$ & $\vdots$ & $\vdots$ \\
    $\infty$ & $H_{2}(X) = X^2 -2$ & $H_{3}(X) = X^3 -6X$. \\
    
    \hline
  \end{tabular}
\end{center}

\end{ex}

\subsection{Modular $L$-functions}
\subsubsection{Proof of Corollary \ref{Mod}}
 Let $f \in S_{k} \left(\Gamma_{0}(N)\right)$ be an even weight newform with real coefficients and write $f(z) = \sum_{n \geq 1} a(n) e^{2 \pi i n z}$.  Assume that $f$ is normalized so that $a(1)=1$. We focus newforms with trivial character.  Define the L-function associated to $f$ by
\begin{equation} \label{modular}
L(f, s) := \sum_{n \geq 1} \frac{a(n)}{n^{s}}
\end{equation}
for ${\rm{Re}}(s) >1 + \frac{k}{2}$.
Define the completed modular $L$-function by 
\begin{equation}
\Lambda(f, s) := N^{\frac{s}{2}} (2 \pi)^{-s} \Gamma(s) L(f, s).
\end{equation}
We have the transformation property
\begin{equation}
f \left( \frac{i}{Ny} \right) = i^{k} \epsilon_{f} N^{\frac{k}{2}} y^{k} f(iy),
\end{equation}
which gives rise to the functional equation 
\begin{equation} \label{Mfunctional}
\Lambda(f, s) = i^{k} \epsilon_{f} \Lambda(f, k-s),
\end{equation}
where $\epsilon_{f} \in \{ \pm 1 \}$ is the eigenvalue of $f$ under the Atkin-Lehner involution.  The completed modular $L$-function $\Lambda(f,s)$ has the required integral representation, the modular properties of $f(z)$ gives a functional equation, and the coefficients are real so $L(f,s)$ is good.

\subsubsection{Derivatives at central values and modular Jensen polynomials}
Similarly to the Dirichlet $L$-function case, the $n$th derivative takes the form
\begin{equation} \label{Mderiv}
\Lambda^{(n)}(f,s)= \frac{1}{2^n} \int_{\frac{1}{\sqrt{N}}}^{\infty} f(iy) \left( N^{\frac{s}{2}} y^{s} + (-1)^{n} i^{k} \epsilon_{f} N^{\frac{k-s}{2}} y^{k-s}  \right) \left( \ln(N) + 2 \ln(y) \right)^{n} \frac{dy}{y}.
\end{equation}
At the central value $s = \frac{k}{2}$ we have
\begin{equation}
\Lambda^{(n)} \left(f, \frac{k}{2} \right) = \frac{1}{2^{n}} \int_{\frac{1}{\sqrt{N}}}^{\infty} f(iy) N^{\frac{k}{4}} y^{\frac{k}{2} -1} (1 + (-1)^{n} i^{k} \epsilon_{f}) \left( \ln(N) + 2 \ln(y) \right)^{n} dy.
\end{equation}
The following theorem gives an arbitrary precision asymptotic formula for these derivatives at central values.
\begin{Thm}
Assume the notation above.  Large $n$ asymptotics for $\Lambda^{(n)} \left(f, \frac{k}{2} \right)$ and $\Xi^{(n)}(f,0)$ is given to all orders by the asymptotic expansion
\begin{align}
F(n) &\sim \frac{\sqrt{2 \pi} N^{\frac{k}{4}} (1 + (-1)^{n} i^{k} \epsilon_{f})}{2^{n+1}} \frac{L^{n+1}}{\sqrt{ \left(1 + \frac{L}{2} \right)n - \left( \frac{k}{2} -1 \right)L^{2}}} \\
&\times e^{\frac{k}{4}(L - \log(N)) - \frac{2n}{L} - \frac{k}{2} +1} \left(1 + \frac{b_{1}}{n} + \frac{b_{2}}{n^2} + \cdots \right) \qquad (n \to \infty), \nonumber
\end{align}
where $L=L(n) \approx 2 \log \left( \frac{n \sqrt{N}}{\log(n \sqrt{N})} \right)$ is the unique solution of the equation \\ $n = \frac{1}{2} \left( \pi e^{\frac{L-\log(N)}{2}} - \frac{k}{2} +1 \right)L$ and each coefficient $b_{k}$ belongs to $\Q(L)$, the first value being $b_{1} = \frac{L^{4} + 9L^{3} + 32L^{2} + 24L +16}{24(L+2)^{3}}$.
\end{Thm}

We have showed that the modular $L$-function $L(f,s)$ and does not have a pole so define
\begin{equation}
\Xi(f, z) := \Lambda \left(f, \frac{k}{2} -iz \right).
\end{equation}
Depending on the sign of the functional equation we define the Taylor coefficients by
\begin{equation}
\Xi_{1} (f,x) = \sum_{n \geq 0} \frac{\gamma_{f}(n)}{n!} x^{n} = \begin{cases} \Xi(i \sqrt{x})  & \text{ if } \ i^{k} \epsilon_{f} = 1 \\
\frac{\Xi(i \sqrt{x})}{\sqrt{x}} & \text{ if } \ i^{k} \epsilon_{f} = -1, \end{cases}
\end{equation}
where
\begin{equation}
\gamma_{f}(n) = \begin{cases} (-1)^{n} \frac{n!}{(2n)!} \cdot \Xi^{(2n)}(0) & \text{ if } \ i^{k} \epsilon_{f} = 1 \\
i^{2n+1} \frac{n!}{(2n+1)!} \cdot \Xi^{(2n+1)}(0) & \text{ if } \ i^{k} \epsilon_{f} =-1.
\end{cases}
\end{equation}
By Theorem \ref{Main} or from the asymptotic expansion above we have that if $d \geq 1$, then$J_{\gamma_{f}}^{d,n}(X)$ is hyperbolic with at most finitely many exceptions $n$.

\subsection{Dedekind zeta-functions}
\subsubsection{Proof of Corollary \ref{Ded}}
The Dedekind zeta-function case will require some setup and notation.  We will mostly follow the notation in \cite{N}.  Let $K$ be a number field of degree $j$ and $\mathcal{O}_{K}$ its ring of integers.  Denote the embeddings by $\sigma_{1}, \dots, \sigma_{r_{1}}, \rho_{1}, \overline{\rho}_{1}, \dots, \rho_{r_{2}}, \overline{\rho}_{r_{2}}$ where there are $r_{1}$ real embeddings and $r_{2}$ pairs of complex embeddings so that $r_{1} + 2r_{2} =j$.  Denote the class group of $K$ by $Cl(K)$.
Let ${\bf{C}} = \prod_{\tau} \C$ and ${\bf{R}} = \left[ \prod_{\tau} \C \right]^{+} = \{ z \in {\bf{C}} : z = \overline{z} \}$ be the Minkowski space of $K$ where $\overline{z} = \overline{(z_{\tau})} = (\overline{z}_{\overline{\tau}})$ is the usual complex conjugation and $\tau$ runs over the $j$ embeddings.  We define the trace and norm by 
\begin{equation}
Tr(z) = \sum_{\tau} z_{\tau} \qquad N(z) = \prod_{\tau} z_{\tau},
\end{equation}
and have a Hermitian scalar product given by
\begin{equation}
\langle x, y \rangle = \sum_{\tau} x_{\tau} \overline{y}_{\tau}.
\end{equation}
We will also require the spaces
\begin{equation}
{\bf{R}}_{\pm} = \left[ \prod_{\tau} \R \right]^{+} = \{ x \in {\bf{R}} : x_{\tau} = x_{\overline{\tau}} \}
\end{equation}
and 
\begin{equation}
{\bf{R}}_{+}^{*} = \left[ \prod_{\tau} \R_{+}^{*} \right]^{+} = \{ x \in {\bf{R}}_{\pm} :\forall \tau \  x_{\tau} >0 \}
\end{equation}
in order to define the two homomorphisms
\begin{align}
| \ | &: {\bf{R}}^{*} \to {\bf{R}}_{+}^{*} \qquad x = (x_{\tau}) \mapsto |x| = (|x_{\tau}|) \\
\log &: {\bf{R}}_{+}^{*} \xrightarrow{\sim} {\bf{R}}_{\pm} \qquad x = (x_{\tau}) \mapsto \log x = (\log x_{\tau}).
\end{align}
Let $\mathfrak{p}= \{ \sigma, \overline{\sigma} \}$ denote a conjugacy class of embeddings (so $\mathfrak{p}$ has one or two elements depending on whether the embedding is real or complex) and observe that there is an isomorphism between ${\bf{R}}_{+}^{*}$ and $\prod_{\mathfrak{p}} \R_{+}^{*}$.  We now have a Haar measure, which we denote by $\frac{dy}{y}$, that corresponds to the product measure $\prod_{\mathfrak{p}} \frac{dt}{t}$ where $\frac{dt}{t}$ is the usual Haar measure on $\R_{+}^{*}$.  We can now define a suitable generalization of the gamma function by
\begin{equation}
\Gamma_{K}(s) = 2^{(1-2s)r_{2}} \Gamma(s)^{r_{1}} \Gamma(2s)^{r_{2}} = \int_{{\bf{R}}_{+}^{*}} N(e^{-y} y^{s}) \frac{dy}{y}.
\end{equation}

The Dedekind zeta-function for $K$ is given by
\begin{equation}
\zeta_{K}(s) = \sum_{\substack{\mathfrak{a} \subset \mathcal{O}_{K} \\ \text{integral}}} N(\mathfrak{a})^{-s}
\end{equation}
for ${\rm{Re}}(s) >1$ where $N(\mathfrak{a}) = [\mathcal{O}_{K} : \mathfrak{a}]$ is the norm of the ideal $\mathfrak{a}$.  For each $B \in Cl(K)$ we define the partial zeta function by
\begin{equation}
\zeta_{B}(s) = \sum_{\substack{\mathfrak{a} \in B \\ \text{integral}}} N(\mathfrak{a})^{-s}.
\end{equation} 
We therefore have
\begin{equation}
\zeta_{K}(s) = \sum_{B \in Cl(K)} \zeta_{B}(s).
\end{equation}
We define the completed partial Dedekind zeta-function by
\begin{align}
\Lambda(B,s) &= |d_{K}|^{\frac{s}{2}} \pi^{-\frac{js}{2}} \Gamma_{K} \left( \frac{s}{2} \right) \zeta_{B}(s) \\
&= \int_{{\bf{R}}_{+}^{*}} g(iy) N(y)^{\frac{s}{2}} \frac{dy}{y} \nonumber
\end{align}
where $d_{K}$ is the discriminant of $K$ and $g$ is some theta function that we will not specify now.  The image of the unit group $\mathcal{O}_{K}^{*}$ under the mapping $| \ |:  {\bf{R}}^{*} \to {\bf{R}}_{+}^{*}$, which we will denote by $|\mathcal{O}_{K}^{*}|$, is contained in the norm-one hypersurface
\begin{equation}
S = \{ x \in {\bf{R}}_{+}^{*} : N(x) =1 \}.
\end{equation}
We obtain a direct decomposition ${\bf{R}}_{+}^{*} = S \times \R_{+}^{*}$ by writing
\begin{equation*}
y = xt^{\frac{1}{j}}, \qquad x = \frac{y}{N(y)^{\frac{1}{j}}}, \qquad t = N(y)
\end{equation*}
for any $y \in {\bf{R}}_{+}^{*}$.  We will need to choose a fundamental domain $F$ for the action of the group
\begin{equation*}
|\mathcal{O}_{K}^{*}|^{2} = \{ |\epsilon|^{2} : \epsilon \in \mathcal{O}_{K}^{*} \}
\end{equation*}
on $S$.  The $\log$ map $\log : {\bf{R}}_{+}^{*} \to {\bf{R}}_{\pm}$ takes $S$ to the trace-zero space
\begin{equation*}
H = \{ x \in {\bf{R}}_{\pm} : Tr(x) = 0 \}
\end{equation*}
and by Dirichlet's unit theorem the group $|\mathcal{O}_{K}^{*}|$ is taken to a complete lattice $G$ in $H$.  We may choose $F$ to be the preimage of any fundamental mesh of the lattice $2G$.  Now using this decomposition we have that
\begin{equation}
\Lambda(B,s) = \int_{0}^{\infty} (f(\mathfrak{a}, t) - f(\mathfrak{a}, \infty)) t^{\frac{s}{2}} \frac{dt}{t}
\end{equation}
where $B$ is the class of $\mathfrak{a}^{-1}$ and 
\begin{equation}
f(\mathfrak{a}, t) = f_{F}(\mathfrak{a}, t) = \frac{1}{w_{K}} \int_{F} \theta(\mathfrak{a}, ixt^{\frac{1}{j}}) d^{*}x.
\end{equation}
In the above equation $w_{K}$ is the number of roots of unity in $K$, $d^{*}x$ is the appropriate Haar measure such that $d^{*}x \times \frac{dt}{t} = \frac{dy}{y}$, and the theta function is defined by
\begin{equation}
\theta(\mathfrak{a}, z) = \sum_{a \in \mathfrak{a}}  e^{\pi i d_{\mathfrak{a}}^{-\frac{1}{j}} \langle az, a \rangle}
\end{equation} 
where $d_{\mathfrak{a}} = |N(\mathfrak{a})|^{2} |d_{K}|$ is the absolute value of the discriminant of $\mathfrak{a}$.  Using the properties of the theta function it is not difficult to show
\begin{equation}
f_{F}\left( \mathfrak{a}, \frac{1}{t} \right) = t^{\frac{1}{2}} f_{F^{-1}}( (\mathfrak{a} \mathfrak{d}_{K})^{-1}, t)
\end{equation}
and 
\begin{equation}
f_{F}(\mathfrak{a}, \infty) = f(\infty) =\frac{2^{r_{1} + r_{2} -1} R(K)}{w_{K}},
\end{equation}
where $F^{-1}$ is again a fundamental domain, $\mathfrak{d}_{K}$ is the different ideal of $K$, and $R(K)$ is the regulator of $K$.  Note that $(\mathfrak{a} \mathfrak{d}_{K})^{-1}$ is the dual lattice of $\mathfrak{a}$ and that $f(\infty)$ does not depend on the fundamental domain or ideal choice so we will supress notation whenever possible.  We now define the completed Dedekind zeta-function by
\begin{align}
\Lambda(K,s) &= \sum_{B \in Cl(K)} \Lambda(B,s) = |d_{K}|^{\frac{s}{2}} \pi^{-\frac{js}{2}} \Gamma_{K} \left(\frac{s}{2} \right) \zeta_{K}(s) \\
&= \frac{2^{r_{1} + r_{2}} R(K) h_{K}}{s(s-1) w_{K}} + \sum_{i=1}^{h_{K}} \int_{1}^{\infty} ( f(\mathfrak{a}_{i}, t) - f(\infty)) \left( t^{\frac{s}{2}} + t^{\frac{1-s}{2}} \right) \frac{dt}{t}, \nonumber
\end{align}
where $h_{K}$ is the class number of $K$ and if $B_{i}, 1 \leq i \leq h_{k}$ are ideal classes, then $B_{i}$ is the class of $\mathfrak{a}_{i}^{-1}$.  This shows that we have the functional equation
\begin{equation}
\Lambda(K, s) = \Lambda(K, 1-s).
\end{equation}
The completed Dedekind zeta-function has suitable integral representation, functional equation, and real coefficients so $\zeta_{K}(s)$ is good.  

\subsubsection{Derivatives at central values and Dedekind Jensen polynomials}
The $n$th derivative of the completed Dedekind zeta-function has the form
\begin{align}
\Lambda^{(n)}(K,s) &= \frac{2^{r_{1} + r_{2}} R(K) h_{K} \cdot n!}{w_{K}} \cdot \frac{(s-1)^{n+1} - s^{n+1}}{s^{n+1}(1-s)^{n+1}} \\
&+ \sum_{i=1}^{h_{K}} \frac{1}{2^{n}} \int_{1}^{\infty} \left( f(\mathfrak{a}_{i}, t) - f(\infty) \right) \left(t^{\frac{s}{2}} + (-1)^{n} t^{\frac{1-s}{2}} \right) \log^{n}(t) \frac{dt}{t}. \nonumber
\end{align}
At the central value $s=\frac{1}{2}$ we have
\begin{align}
\Lambda^{(n)} \left(K, \frac{1}{2} \right) &= \frac{2^{r_{1} + r_{2} +n +1} R(K) h_{K} ((-1)^{n+1} -1) n!}{w_{K}} \\
&+ \sum_{i=1}^{h_{K}} \frac{1}{2^{n}} \int_{1}^{\infty} \left( f(\mathfrak{a}_{i}, t) - f(\infty) \right) t^{\frac{1}{4}} (1+(-1)^{n}) \log^{n}(t) \frac{dt}{t}. \nonumber
\end{align}
In order to state the asymptotic expansion we need to find the first nonzero coefficient of each $f(\mathfrak{a}_{i} t)$.  Let $\epsilon$ be a unit with norm $1$, then the smallest nonzero exponent in $f(\mathfrak{a}_{i}, t)$ is given by
\begin{equation*}
m_{\mathfrak{a}_{i}} = {\rm{min}} \{ \langle a \epsilon, a \rangle : a \in \mathfrak{a}_{i}, a \neq 0 \}.
\end{equation*}
Let 
\begin{equation*}
M_{\mathfrak{a}_{i}} = \# \{a \in \mathfrak{a}_{i} : \langle a \epsilon, a \rangle = m_{\mathfrak{a}_{i}} \},
\end{equation*}
then the expansion of $f(\mathfrak{a}_{i}, t)$ begins
\begin{align}
f(\mathfrak{a}_{i}, t) &= f(\infty) + \frac{2^{r_{1} + r_{2} -1} R(K)}{w_{K}} M_{\mathfrak{a}_{i}} e^{-\pi m_{\mathfrak{a}_{i}} \left( \frac{t}{d_{\mathfrak{a}_{i}}} \right)^{\frac{1}{j}}} + \cdots.
\end{align}
We will let $C_{i} = \frac{2^{r_{1} + r_{2} -1} R(K)}{w_{K}} M_{\mathfrak{a}_{i}}$ and 
\begin{equation}
F_{i}(n) = \frac{1}{2^{n}} \int_{i}^{\infty} \left( f(\mathfrak{a}_{i}, t) - f(\infty) \right) t^{-\frac{3}{4}} (1+(-1)^{n}) \log^{n}(t) dt
\end{equation}
in order to simplify the next theorem.
\begin{Thm} Assume the notation above, then we have
\begin{equation}
\Lambda^{(n)} \left(K, \frac{1}{2} \right) = \frac{2^{r_{1} + r_{2} +n +1} R(K) h_{K} ((-1)^{n+1} -1) n!}{w_{K}} + \sum_{i=1}^{h_{K}} F_{i}(n)
\end{equation} 
and $F_{i}(n)$ is given to all orders by the asymptotic expansion
\begin{align}
F_{i}(n) & \sim \frac{C_{i} \sqrt{2 \pi} (1 + (-1)^{n})}{2^{n}} \frac{L_{i}^{n+1}}{\sqrt{n \left( 1 + \frac{L_{i}}{j} \right) - \frac{3}{4j} L_{i}^{2}}} \\
&\times e^{\frac{L_{i}}{4} - \frac{jn}{L_{i}} + \frac{3j}{4}} \left(1 + \frac{b_{i,1}}{n} + \frac{b_{i,2}}{n^2} + \cdots \right) \qquad (n \to \infty), \nonumber
\end{align}
where $L_{i}=L_{i}(n) \approx j \log \left(\frac{n}{\log(n)} \right)$ is the unique solution of the equation $n = \left(\frac{m_{\mathfrak{a}_{i}} d_{\mathfrak{a}_{i}}^{-\frac{1}{j}}}{j} \pi e^{\frac{L_{i}}{j}} + \frac{3}{4} \right) L_{i}$ and each coefficient $b_{i,k}$ belongs to $\Q(L_{i})$.
\end{Thm}

We have shown that $\zeta_{K}(s)$ is good so define 
\begin{equation}
\Xi(z) := \left(-z^{2} - \frac{1}{4} \right) \Lambda \left(K, \frac{1}{2} -iz \right)
\end{equation}
and 
\begin{equation}
\Xi_{1}(x) := \Xi(i \sqrt{x}) = \sum_{n \geq 0} \frac{\gamma_{K}(n)}{n!} x^{n}
\end{equation}
where the Taylor coefficients are given by
\begin{equation}
\gamma_{K}(n) = (-1)^{n} \frac{n!}{(2n)!} \cdot \Xi^{(2n)}(0).
\end{equation}
The derivatives $\Xi^{(2n)}(0)$ are given by
\begin{equation*}
\Xi^{(2n)}(0) = (-1)^{n} \sum_{i=1}^{h_{K}} \frac{8 \binom{2n}{2} F_{i}(2n-2) - F_{i}(2n)}{4}
\end{equation*}
and so we can use the above asymptotic expansion above or Theorem \ref{Main} to show that if $d \geq 1$, then $J_{\gamma_{K}}^{d,n}(X)$ is hyperbolic with at most finitely many exceptions $n$.

\end{document}